\documentclass{amsart}

\usepackage{amsmath,amssymb,amsthm,graphicx, paralist, appendix}
\usepackage{hyperref}
\usepackage{varioref} 
\usepackage{subfigure}
\usepackage{amsrefs}
\usepackage[arrow, matrix, curve]{xy}
\usepackage{stmaryrd} 
\usepackage{color} 

\def\co{\colon\thinspace}

\newcommand{\id}{\mbox{\rm id}}
\newcommand{\e}{\mathrm{e}}

\newcommand{\C}{\mathbb{C}}
\newcommand{\N}{\mathbb{N}}
\newcommand{\R}{\mathbb{R}}
\newcommand{\Z}{\mathbb{Z}}

\newcommand{\myparskip}{\parskip 4pt}

\newtheorem{thm}{Theorem}
\newtheorem*{thm*}{Theorem}
\newtheorem{lem}[thm]{Lemma}
\newtheorem{prop}[thm]{Proposition}

\newtheorem*{claim*}{Claim}

\theoremstyle{definition}

\newtheorem*{rem*}{Remark}
\newtheorem*{rems*}{Remarks}
\newtheorem*{remdefn*}{Remark/Definition}
\newtheorem{defn}{Definition}
\newtheorem*{defn*}{Definition}
\newtheorem*{ex*}{Example}

\newtheorem*{que*}{Question}

\begin{document}
\title[]{Open book decompositions of fibre sums in contact topology}
\author{Mirko Klukas}
\address{Mathematisches Institut, Universit\"at zu K\"oln, Weyertal 86--90, 50931 K\"oln, Germany}
\email{mklukas@math.uni-koeln.de}
\date{\today}
\maketitle
\begin{abstract}
In the present paper we describe compatible open books for the
fibre connected sum along binding components of open books, as well as for the
fibre connected sum along multi-sections of open books.
As an application the first description provides simple ways of constructing open books supporting all tight contact structures on $T^3$, recovering a result by van Horn-Morris, as well as an open book supporting the result of a Lutz twist along a binding component of an open book, recovering a result by Ozbagci--Pamuk.
\end{abstract}
%
%
\section*{Introduction}
%
%
According to a theorem of Alexander~\cite{Ale-SystemsOfKnottedCurves} every closed oriented
$3$-manifold admits a so-called \textit{open book decomposition}.
While it had been known for almost 40
years that open books carry a natural contact structure \cite{MR0375366}, at
the beginning of the millennium it turned out that this was just one fragment
of a much deeper correlation. As was observed by Giroux \cite{MR1957051},
contact structures in dimension $3$ are of purely topological nature: he
established a one-to-one correspondence between isotopy classes of contact
structures and open book decompositions up to positive stabilisation.
Ever since it is of interest to recover properties and constructions of contact structures in the language of open books.
\parskip 0pt

In the present paper we approach how the contsruction of the \textit{fibre connected sum} affects an underlying open book structure of the original manifold under certain assumptions. Fibre connected sums in the contact setting recently drew some attention appearing in Wendl's~\cite{2010arXiv1009.2746W} notion of \textit{planar torsion}, an obstruction for strong fillability generalising overtwistedness and Giroux torsion. In essence, a contact manifold admits planar torsion if it can be written as the \textit{binding sum}, i.e.~the fibre connected sum along binding components of an open book, of a non-trivial ($\geq 2$) number of open books, one of which has planar pages.
\myparskip

The present paper, suppressing the preliminaries, splits into two parts, \S\ref{sec:binding sum} and \S\ref{sec:multisections}, which can be read independently. The main results are descriptions of compatible open books for:
\begin{list}{}{}
		\item[I.]  The fibre connected sum along binding components of open books -- cf.\ Theorem~\ref{prop:open book supporting the binding sum} on page~\pageref{prop:open book supporting the binding sum}.
		\item[II.]  The fibre connected sum along multi-sections of open books -- cf.\ 
Theorem~\ref{prop:multisections} on page~\pageref{prop:multisections}.
\end{list}
As an application the first description provides simple ways of constructing open books supporting all tight contact structures on $T^3$, recovering a result by van Horn-Morris~\cite{vanHornMorris-Thesis}, as well as an open book supporting the result of a Lutz twist along a binding component of an open book, recovering a result by Ozbagci--Pamuk~\cite{2009arXiv0905.0986O}.
%
\subsection*{Acknowledgements}
%
The results presented in this paper are part of my thesis~\cite{Klukas-Thesis}.
I want to thank my advisor Hansj\"org Geiges for introducing me to the world of
contact topology and for all the helpful discussions, especially for initializing 
the first part of the present paper. 
I deeply thank John Etnyre for many inspiring conversations and, in particular,
for initializing the second part of this paper.
\parskip 0pt 
 
The research for the second part of this paper took place during a research visit at the
Georgia Institute of Technology, Atlanta GA, USA, under the supervision of John
Etnyre. I want to thank GaTech and John Etnyre for their hospitality. This research
visit was additionally supported by the DAAD (German Academic Exchange Service). Overall
the author was supported by the DFG (German Research Foundation) as fellow of the 
graduate training programm \textit{Global structures in geometry and analysis} 
at the Mathematics Department of the University of Cologne, Germany.
%
%
\section{Preliminaries}
%
%
\subsection{Open books}
An \textbf{open book decomposition} of a $3$-dimensional manifold $M$ is a pair
$(B,\pi)$, where $B$ is a disjoint collection of embedded circles, in $M$, called the
  \textbf{binding} of the open book and
  $\pi\co M\setminus B \to S^1$ is a (smooth, locally trivial)
  fibration such that 
  each fibre $\pi^{-1}(\varphi)$, $\varphi\in S^1$, corresponds to the interior
  of a compact hypersurface $\Sigma_\varphi \subset M$ with
  $\partial\Sigma_\varphi = B$.
  The hypersurfaces $\Sigma_\varphi$, $\varphi \in S^1$, are called the
  \textbf{pages} of the open book.
\parskip 0pt

In some cases we are not interested in the exact position of the binding or the pages of an open book decompositon inside the ambient space. Therefore, given an open book decomposition $(B,\pi)$ of a $3$-manifold $M$, we could ask for the relevant data to remodel the ambient space $M$ and its underlying open books structure $(B,\pi)$, say up to diffemorphism. This leads us to the following notion.
\parskip 0pt

An \textbf{abstract open books} is a pair $(\Sigma,\phi)$, where $\Sigma$ is a compact surface with non-empty boundary $\partial \Sigma$, called the \textbf{page} and $\phi\co\thinspace \Sigma \to \Sigma$ is a diffeomorphism equal to the identity near $\partial \Sigma$, called the \textbf{monodromy} of the open book.
Let $\Sigma(\phi)$ denote the mapping torus of $\phi$, that is, the quotient space obtained from $\Sigma \times [0,1]$ by identifying $(x,1)$ with $(\phi(x),0)$ for each $x \in \Sigma$. Then the pair $(\Sigma,\phi)$ determines a closed manifold $M_{(\Sigma,\phi)}$ defined by 
\begin{equation}
\label{eqn:abstract open book}
			M_{(\Sigma,\phi)} := \Sigma(\phi) \cup_{\id} (\partial \Sigma \times D^2),
\end{equation}
where we identify $\partial \Sigma(\phi) = \partial \Sigma \times S^1$ with $\partial (\partial \Sigma \times D^2)$ using the identity map.
Let $B \subset M_{(\Sigma,\phi)} $ denote the embedded link $\partial \Sigma \times \{0\}$. Then we can define a fibration $\pi\co M_{(\Sigma,\phi)}\setminus B \to S^1$ by
\[
	\left.
	 \begin{array}{l}
				\lbrack x,\varphi \rbrack  \\
				 \lbrack \theta, r\e^{i\pi\varphi} \rbrack  
		\end{array} \right\}
							\mapsto [\varphi],
\]
where we understand $M_{(\Sigma,\phi)}\setminus B$ as decomposed in (\ref{eqn:abstract open
book}) and $[x,\varphi] \in \Sigma(\phi)$ or $ [\theta, r\e^{i\pi\varphi}] \in
\partial\Sigma \times D^2 \subset \partial\Sigma \times \C$ respectively. Clearly $(B,\pi)$ defines an open book decomposition of $M_{(\Sigma,\phi)}$.
\parskip 0pt

On the other hand, an open book decomposition $(B,\pi)$ of some $3$-manifold
$M$ defines an abstract open book as follows: identify a neighbourhood of $B$
with $B \times D^2$ such that $B = B\times\{0\}$ and such that the fibration on
this neighbourhood is given by the angular coordinate, $\varphi$ say, on the
$D^2$-factor. We can define a $1$-form $\alpha$ on the complement $M \setminus (B \times
D^2)$ by pulling back $d\varphi$ under the fibration $\pi$, where this time we
understand $\varphi$ as the coordinate on the target space of $\pi$.
The vector field $\partial \varphi$ on $\partial\big(M \setminus (B \times D^2)
\big)$ extends to a nowhere vanishing vector field $X$ which we normalise by
demanding it to satisfy $\alpha(X)=1$. Let $\phi$ denote the time-$1$ map of the
flow of $X$. Then the pair $(\Sigma,\phi)$, with $\Sigma = \overline{\pi^{-1}(0)}$, defines an abstract open book such that
$M_{(\Sigma,\phi)}$ is diffeomorphic to $M$.
\subsubsection{Examples} 
\label{ex:open books for S3} 
Understand $S^3$ as the unit sphere in $\C^2$, i.e.~as the subset of $\C^2$ given by 
\[
			S^3 = \{ (z_1,z_2) \in \C^2 \co |z_1|^2 + |z_2|^2 = 1 \}.
\]
We give three examples of open book decompositions of $S^3$:
\myparskip

\begin{list}{}{}
	\item[(1)] Set $B = \{(z_1,z_2)\in S^3 \co z_1 = 0  \}$. Note that $B$ is an unknotted circle in $S^3$. Consider the fibration
	\[
				\pi \co S^3 \setminus B \to S^1\subset \C, \ (z_1,z_2) \mapsto \frac{z_1}{|z_1|}.
	\]
	In polar coordinates this map is given by $(r_1 \e^{i\varphi_1},r_2 \e^{i\varphi_2}) \mapsto \varphi_1$.
	Observe that $(B,\pi)$ defines an open book decomposition of $S^3$ with pages diffeomorphic to $D^2$ and trivial monodromy.
	\item[(2)] Set $B_+ = \{(z_1,z_2)\in S^3\co z_1 z_2 = 0  \}$.
	Observe that $B_+$ describes the positive Hopf link. Consider the fibration
	\[
				\pi_+ \co S^3 \setminus B_+ \to S^1\subset \C, 
				\ (z_1,z_2) \mapsto \frac{z_1z_2}{|z_1z_2|}.
	\]
	One can show that $(\pi_+ ,B_+)$ defines an open book decomposition of $S^3$ 
	with annular pages and monodromy given by a left-handed Dehn twist along the core of the annulus.
	\item[(3)] Set $B_- = \{(z_1,z_2)\in S^3\co z_1 \overline{z_2} = 0  \}$.
	Observe that $B_-$ describes the negative Hopf link. Consider the fibration
	\[
				\pi_- \co S^3 \setminus B_- \to S^1\subset \C, 
				\ (z_1,z_2) \mapsto \frac{z_1\overline{z_2}}{|z_1\overline{z_2}|}.
	\]
	One can show that $(\pi_+ ,B_+)$ defines an open book decomposition of $S^3$ 
	with annular pages and monodromy given by a right-handed Dehn twist along the core of the annulus.
\end{list}
\subsection{Stabilisations of open books}
Let $\Sigma$ be a compact surface with non-empty boundary and $\phi\co \Sigma \to \Sigma$ a diffeomorphism equal to the identity near $\partial \Sigma$.
Suppose further we are given a properly embedded arc $a \subset \Sigma$. 
The \textbf{positive (negative) stabilisation} of the abstract open book $(\Sigma,\phi)$ is the abstract open book obtained by adding a $1$-handle to the original page $\Sigma$ along the endpoints of $a$, and changing the monodromy by composing it with a right- (left-) handed Dehn twist along the simple closed curve obtained by the union of $a$ and the core of the $1$-handle.
The open books described in parts (2) and (3) of the preceding example are instances of a positive and negative stabilisation respectively of the open book described in the first part.
\parskip 0pt

A positive and negative stabilisation respectively can be understood as a
suitable connected sum with the open book described in part (2) or part (3) of
the preceding example. This viewpoint will draw more attention once we introduced the interplay of open books and contact structures, cf.\ \S\ref{sec:compatibility} below. To be more precise: a positive stabilisation
of an open book does not change the underlying contact structure, whereas a
negative stabilisation turns it into an overtwisted one.
\subsection{Compatibility}
\label{sec:compatibility}
A positive contact structure $\xi = \ker \alpha$ and an open book decomposition $(B,\pi)$ of $M$ are said to be \textbf{compatible} with each other, if the $2$-form $d\alpha$ induces a symplectic form on each page, defining its positive orientation, and the $1$-form $\alpha$ induces a positive contact form on $B$.
\parskip 0pt

In dimension equal to $3$ it can be shown that any two contact structures supported by the same open book decomposition are in fact contact isotopic (cf.~\cite{MR2249250}). 
The open books described in parts (1) and (2) of Example \ref{ex:open books for S3} above support the standard contact structure $\xi_{st}$ on $S^3$, whereas part (3) supports the overtwisted contact structure $\xi_{1}$ which is obtained by a Lutz twist along a transverse unknot $U \subset (S^3,\xi_{st})$ with self-linking number $-1$.
%
%
\subsection{The fibre connected sum}
\label{sec:def fibre sum}
%
%
Let $K \subset (M,\xi)$ be a positive transverse knot sitting in a contact $3$-manifold $(M,\xi)$. We may identify a neighbourhood of $K$ with an $\varepsilon$-neighbourhood $N_\varepsilon \subset S^1 \times \R^2$, where $K = S^1 \times \{0\}$. Then, with $S^1$-coordinate $\theta$, polar
coordinates $(r,\varphi)$ on $\R^2$, and for a suitable $\varepsilon > 0$, the contact structure
\[
			d\theta + r^2 \thinspace d\varphi=0
\]
provides a model for the above neighbourhood of $K$.
Let $M_K= M \setminus N_\delta$ denote the complement of a $\delta$-neighbourhood $N_{\delta} \subset N_\varepsilon$, with $0 < \delta < \varepsilon$, where we change the contact structure as follows. Replace the contact structure $\xi$ over $N_\varepsilon \setminus N_{\delta}$ by the kernel of the contact $1$-form
$d\theta + f(r) \thinspace d\varphi$, where $f: [\delta, \infty] \to \R$ is a function that equals $r^2$ away from $\delta$, satisfies $f' > 0$, $f'(\delta) = 1$ and $f(\delta)=0$. 
We will refer to $M_K$ as obtained by \textbf{blowing up $K$}.
The inverse operation of blowing up will be referred to as \textbf{collapsing}. 
\parskip 0pt

Suppose now we are given a pair of positive transverse knots $K_0,K_1 \subset (M,\xi)$ wich are endowed with a framing. Let $M_{K_0,K_1}$ denote the result of blowing up each of the knots $K_0$ and $K_1$. Each of the boundary tori associated to $K_0$ and $K_1$ respectively admits a natural identification with $\R^2/\Z^2$ by sending the meridian to the the $x$-axis and the framing-direction to the $y$-axis.
The \textbf{fibre connected sum $\#_{K_0,K_1}(M,\xi)$} is the closed oriented manifold
given as the quotient space
\[
			\#_{K_0,K_1}(M,\xi) := (M_{K_0,K_1}) / \sim
\]
where we identify the boundary tori with respect to the gluing map sending $(x,y)$ to $(-x,y)$. 
%
%
\section{An open book supporting the binding sum}
\label{sec:binding sum}
%
%
In the previous section we introduced the fibre connected sum along a pair of positively transverse knots in a contact manifold. In the present section we proceed by considering two special cases, the fibre connected sum along binding components of open books and the fibre connected sum along sections of open books. Throughout the whole section let $(M,\xi)$ be a closed, not necessarily connected, contact $3$-manifold supported by an open book $(\Sigma,\phi)$. Let $B \subset M$ denote the embedded binding of the open book.
\parskip 0pt

Suppose we have chosen the transverse knots $K_0$ and $K_1$ to be components of
the binding of the open book decomposition $(\Sigma,\phi)$. Since the pages
induce a natural framing for $K_0$ and $K_1$ respectively we can think of it as
the zero-framing and hence can measure all other trivialisations relative to
it. Note that performing the fibre connected sum with framings
$m_1,m_2 \in \mathbb{Z}$ equals the the result of performing the fibre
connected sum with framings $\widetilde m_1 = m_1 + m_2$ and $\widetilde m_2 =
0$. So in the following we just fix one framing assuming the other one to be
zero.
The result of performing the fibre connected sum along two binding components
$K_0$ and $K_0$ with framing $m \in \mathbb{Z}$ will be referred to as
\textbf{binding sum along $K_0$ and $K_1$} and will be denoted by
\[
	\boxplus_m (\Sigma,\phi).
\]
\parskip 0pt

Now suppose $K_0$ and $K_1$ are positive transverse knots intersecting  every
page transversely and exactly once. We will refer to such a knot as a
\textbf{section of the open book}, since it induces a section of the fibration
$M \setminus B \to S^1$. Again understand these knots as endowed with a
framing.
By nature of the sections we can embed the normal bundle $N_i$, $i=0,1$, such
that each fibre corresponds to a disc neighbourhood $D_i \subset \Sigma$ of the intersection
point $\{p_i\} = K_i \cap \Sigma$. In the following we will see that fibre connected sums of
this kind are nicely adapted to the underlying open book decomposition.
\parskip 0pt

The new fibre is obtained by replacing $D_0 \cup D_1$ by $[-1,1] \times S^1$.
However, the change of monodromy is less obvious. To see how the monodromy
changes, consider a vector field transverse to the fibres in $M$ with $K_0$
and $K_1$ as closed orbits such that the return map $h$ on a fibre $\Sigma$
fixes a disc neighbourhood $D_i$ of each $\Sigma \cap K_i$ and such that
closed orbits close to $K_0$ and $K_1$ represent the trivialisations of the
sections. The new monodromy is equal to $h$ on $\Sigma \setminus (D_0 \cup
D_1)$ and the identity on $[-1,1] \times S^1$.
\parskip 0pt

For our purposes it will be sufficient just to consider \textbf{trivial
sections}, that is, sections corresponding to a single fix point $p \in \Sigma$
of the monodromy of a given abstract open book $(\Sigma,\phi)$. In this
case we obtain natural trivialisations of the normal bundles given by a
parallel copy of the knot corresponding to a nearby point. Furthermore we can
assume the given monodromy $\phi$ to be the identity on $D_0 \cup D_1$. So by
the observations above, the new monodromy will be given by $\phi$ on $\Sigma
\setminus (D_0 \cup D_1)$ and the identity on $[0,1] \times S^1$.
\begin{lem}
\label{lem:sections are nice}
Let $K$ be a section of an open book $(B,\pi)$ supporting a contact $3$-manifold $(M,\xi)$. Then there is another contact structure $\xi'$ which is still supported by $(B,\pi)$ and such that the intersection point of $K$ and any page $\Sigma_\varphi$, $\varphi \in S^1$, corresponds to an elliptic singularity of the characteristic foliation $(\Sigma_\varphi)_{\xi'}$. The analogous statement holds for multi-sections of open books as defined in
Section \ref{sec:multisections}.
\end{lem}
\begin{proof}
Fix a page $\Sigma = \Sigma_0$ of $(B,\pi)$ and let $N(B)$ denote a neighbourhood of the binding disjoint from $K$. Furthermore let $p = \Sigma \cap K$ denote the transverse intersecion of $K$ and $\Sigma$.
Identify the complement $M\setminus N(B)$ of $N(B)$ with the mapping torus $\Sigma(\phi)$ of a suitable monodromy map $\phi\co \Sigma \to \Sigma$, i.e.\ we have
\[
			M\setminus N(B) \cong \Sigma \times [0,1] / \sim,
\]
where we identify $(x,1)$ with $(\phi(x),1)$ for all $x \in \Sigma$. Note that we may assume that we have chosen $\phi$ in such a way that, with respect to the above identification, $\{ p \} \times [0,1]$ descends to $K$. Moreover we can assume $\phi$ to fix a little neighbourhood of $p$. Hence it remains to show that there is a contact structure on $\Sigma(\phi)$ satisfying the desired property for this particular case.
Following the construction of Thurston and Winkelnkemper~\cite{MR0375366} all we have to do is choose an exact volume form $d\beta$ on $\Sigma$ such that its dual vector field $Y$ points outwards along $\partial\Sigma$ and has an elliptic singularity at $p$.
\end{proof}
\begin{lem}
The contact manifold resulting from the (contact) fibre connected sum along a section
of an open book is compatible with the corresponding open book.
The analogous statement holds for multi-sections of open books as defined in
Section \ref{sec:multisections}.
\end{lem}
\begin{proof}
According to Lemma~\ref{lem:sections are nice} We can assume, by applying an isotopy of $\xi$,
that the intersections of $K_i$ with the pages $\Sigma_\theta$ correspond to elliptic singularities of the characteristic foliation $(\Sigma_\theta)_\xi$. If we perform the fibre connected sum, the resulting contact structure $\xi'$ and the open
book $(\Sigma',\phi')$ are related as follows. As
observed in the above discussion the new fibre $\Sigma'$ is obtained by
replacing $D_0 \cup D_1 \subset \Sigma$ by $[-1,1] \times S^1$. Since the
origin $0 \in D_i$ corresponds to an elliptic singularity of the characteristic foliation
$(D_i)_\xi$ it gives rise to a closed leaf in the characteristic foliation
$([-1,1] \times S^1)_{\xi'}$ corresponding to the core $\{0\}\times S^1$ of the
annulus. Outside this curve the characteristic foliation agrees with the
foliations on $D_i \setminus \{0\}$.
\parskip 0pt

The new contact structure $\xi'$ can be isotoped to be arbitrarily close
(as oriented plane fields), on compact subsets of the pages, to the tangent
planes to the pages of the open book in such a way, that after some point
in the isotopy the contact planes are transverse to $B'$ and transverse to the
pages of the open book in a fixed neighbourhood of $B'$ (because this holds for 
the original open book $(\Sigma,\phi)$). Hence, according to
\cite{MR2249250}*{Lemma 3.5}, the contact structure $\xi'$ is supported by the
open book $(\Sigma',\phi')$.
\end{proof}
\begin{rem*}
An open book decomposition can be understood as the
boundary of an \textit{achiral Lefschetz fibration}. In a similar fashion the
fibre sum along sections corresponds to the boundary of a \textit{broken achiral Lefschetz
fibration}, see \cite{MR2350472} for reference.
\end{rem*}
Given a simple closed curve $\alpha$ on 
some surface $\Sigma$ we denote by $\tau_\alpha$ and $(\tau_\alpha)^{-1}$ respectively the right- and left-handed Dehn twist along $\alpha$.  When we deal with the 
concatenation of Dehn twists we sometimes omit the concatenation symbol ``$\circ$'' to simplify the notation. In this fashion it makes sense to consider $n$th-powers $(\tau_\alpha)^n$ of Dehn twists, where the zero power $(\tau_\alpha)^0$ is defined to be the identity map.
\parskip 0pt

Let $K \subset \partial\Sigma$ denote a boundary component of the page $\Sigma$ provided with some framing $m \in \Z$ and let $K' \subset \Sigma$ denote the transverse knot indicated by the black dot in Figure~\ref{fig:navel} which we understand $K'$ as zero-framed. The framings are understood as measured with respect to their respective natural zero-framings as explained above.
\begin{defn}
We will refer to $(K,m)$ as
\textbf{admitting a navel} if the monodromy near the boundary component is given
by $\tau_\alpha\tau_\beta^{-1}\tau_\gamma^{m-1}$, where the curves $\alpha,
\beta, \gamma \subset \Sigma$ are given as in Figure~\ref{fig:navel}. The transverse knot $K'
= (K',0)$ indicated by the black dot in Figure~\ref{fig:navel} will be referred
to as \textbf{core of the navel} corresponding to $(K,m)$.
\end{defn}
Observe that we can change every boundary component into a navel, since the monodromy
$\tau_\alpha\tau_\beta^{-1}\tau_\gamma^{m-1}$ is isotopic to the identity.
In the following proposition we will express the binding sum as the fibre
sum along the core of its corresponding navel.
\begin{figure}[h] 
\center
\includegraphics{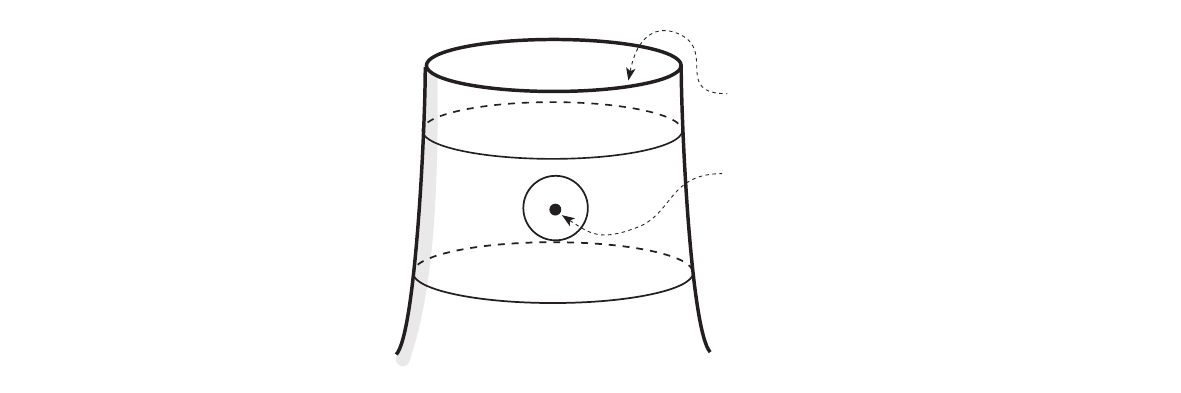}
\put(-130,82){$K$}
\put(-130,60){$K^\prime$}
\put(-230,70){$\alpha$}
\put(-234,30){$\beta$}
\put(-200,50){$\gamma$}
\caption{Binding component admitting a navel.}
\label{fig:navel}
\end{figure}
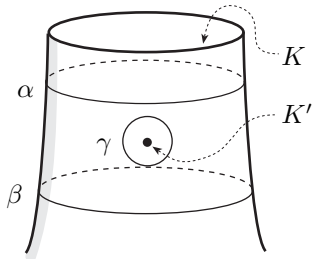
\begin{thm}
\label{prop:open book supporting the binding sum}
Let $K \subset \partial\Sigma$ be a binding component provided with a framing
$m \in \Z$. Then the framed knot $(K,m)$ is transversely isotopic to the
corresponding core $(K',0)$ of its navel. In consequence the result of
performing the binding sum with framing $m \in \mathbb{Z}$ along two binding
components $K_0,K_1\subset \Sigma$ corresponds to the fibre sum along the cores
$K'_0,K'_1$ of their corresponding navels (cf.\ also Figure~\ref{fig:binding sum}).
\end{thm}
\begin{figure}[h] 
\center
\includegraphics{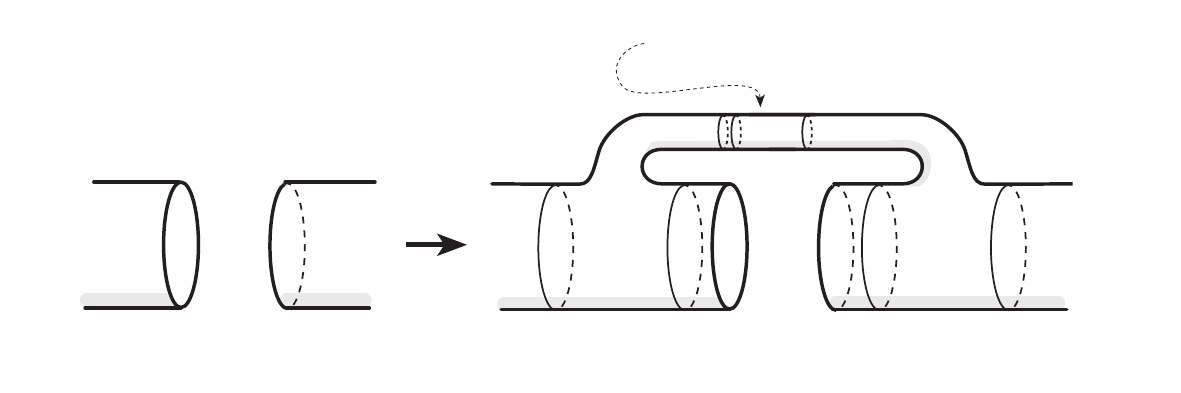}
\put(-280,37){$\boxplus_m$}
\put(-124,75){$\dots$}
\put(-149,106){\textit{\small{$|m-2|$ Dehn twists of sign $\pm$,}}}
\put(-149,96){\textit{\small{depending on the sign of $(m-2)$}}}
\put(-147,14){$+$}
\put(-184,14){$-$}
\put(-90,14){$+$}
\put(-55,14){$-$}
\caption{An open book supporting the binding sum.}
\label{fig:binding sum}
\end{figure}
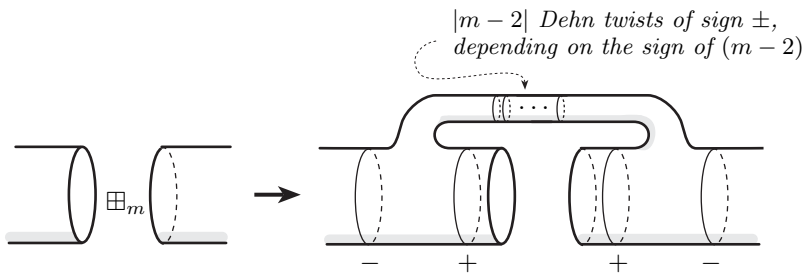
\begin{proof}
Identify a neighbourhood of the binding component $K\subset \partial\Sigma$ with an $\varepsilon$-neighbourhood $N_\varepsilon \subset S^1 \times \R^2$, where $K = S^1 \times \{0\}$. Then, with $S^1$-coordinate $\theta$, polar
coordinates $(r,\varphi)$ on $\R^2$, and for a suitable $\varepsilon > 0$, the contact structure
\[
			d\theta + r^2 \thinspace d\varphi=0
\]
provides a model for the above neighbourhood of $K$. Moreover we can assume
that over this neighbourhood the pages are given by the
preimages of the projection on the angular-coordinate $\varphi$, i.e.~the
closure of every page can be described as $A_\varphi = S^1  \times [0,\varepsilon] \times
\{\varphi\}$ for some appropriate $\varphi \in S^1$.
\parskip 0pt

We will now apply the first part $\tau_\alpha \tau_\beta^{-1}$ of the monodromy
of the navel (the twists $\tau_\gamma^{m-1}$ just take care of the framings,
but we will come to that later). 
Let $S^1 \times [\delta,\varepsilon] \times S^1$ denote the complement of a $\delta$-neighbourhood $N_\delta$ of $K$ in $N_\varepsilon$ for some small $0 < \delta < \varepsilon$.
Consider the map $\phi \co S^1 \times
[\delta,\varepsilon] \to S^1 \times [\delta,\varepsilon]$ defined by
\[
					\phi(\theta,r) := (\theta + \thinspace h(r),r),
\]
where $h:[\delta,\varepsilon] \to [0,1]$ is the function satisfying the following properties:
\begin{list}{$\bullet$}{}
	\item $h(r) = 0$ for $r$ near $\delta$ and near $\varepsilon$,
	\item $h(r) = 1$ on an interval containing $\frac{\varepsilon + \delta}{2}$,
	\item $h'(r) \geq 0$ for $r< \frac{\varepsilon + \delta}{2}$ and
	\item $h'(r) \leq 0$ for $r> \frac{\varepsilon + \delta}{2}$.
\end{list}
Note that with respect to the identification $S^1 \equiv \R/\Z$ the map $\phi$
is indeed well-defined and observe that $\phi$ equals $\tau_\alpha \tau_\beta^{-1}$
and is isotopic to the identity. Consider the corresponding mapping torus
$A(\phi)$, that is
\[
				A(\phi)  = \big( S^1 \times [\delta,\varepsilon] \times [0,1] \big) / \sim_{\phi},
\]
where we identify $(\theta + h(r),r,1)$ with $(\theta + h(r),r,0)$ for each
$(\theta,r) \in S^1 \times [\delta,\varepsilon]$. Following the construction of
Thurston-Winkelnkemper \cite{MR0375366} we can endow $A(\phi)$ with the
contact structure $\xi'$ given by the kernel of the contact $1$-form
\[
		 \big( (1-\varphi)\thinspace d\theta + \varphi\thinspace \phi^* d\theta \big) + r^2\thinspace d\varphi
\]
(actually this defines a contact structure on $S^1 \times [\delta,\varepsilon] \times
[0,1]$ that descends to a contact structure on $A(\phi)$).
Observe that, since $\phi$ is isotopic to the identity, $(A(\phi),\xi')$ and $(S^1 \times [\delta,\varepsilon] \times S^1,\xi)$ are contactomorphic under a contactomorphism keeping little neighbourhoods of
the boundary fixed.
The space $A(\phi)$ is foliated by tori $T_r$ of the form
\begin{equation}
\label{eqn:T_r}
				T_r = \big( S^1 \times \{ r \} \times [0,1]  \big) /_{\sim_{\phi}},
\end{equation}
where we identify $(\theta,r,1)$ with $(\theta + h(r),r,0)$ for each $\theta
\in S^1$. We can also understand these tori as the quotient of $\R^2$ and the
lattice spanned by $(1,0)$ and $(h(r),1)$ (in the same manner as we understand
$S^1 \times S^1$ as $\R^2 /\Z^2$). The characteristic foliation $(T_r)_\xi$ of
each torus is given by linear curves of slope $s(T_r) = -\frac{1}{r^2}$, where
we measure the slope with respect to the identification as above. Hence any
closed curve $c$ on $T_r$ describes a transverse knot as long as the slope of
$\dot c$ does not equal $-\frac{1}{r^2}$. Now let $K^+$ be the positive,
transverse push-off of $K$, i.e.~$K ^+$ is a linear curve on $T_\delta$ of slope
$+1$.
\parskip 0pt

We are now going to define an isotopy $K_s$ of transverse knots connecting
$K^+$ with the core of the navel $K'$. For $s\in[\delta,\frac{\varepsilon + \delta}{2}]$
let $c_s:[0,1] \to [0,1]\times [0,1]$ denote the family of embedded curves with
the following properties:
\begin{list}{$\bullet$}{}
	\item $c_s(0)=(h(s),0)$,
	\item $c_s(1)=(1,1)$,
	\item $\dot c_s\geq0$ and
	\item $\dot c_s(0) = \dot c_s(1) = \pm\infty$.
\end{list}
Observe that each of the curves $c_s$ gives rise to a closed curve $K_s$ on
$T_s$ (cf.~Equation (\ref{eqn:T_r}) above). In particular these knots are
transverse, since we have $\dot c_s \geq 0 > s(T_s)$. Furthermore we have
$K_\delta = K^+$ and $K_{\frac{\varepsilon + \delta}{2}} = K'$.
\parskip 0pt

Let us see what happens to the framing of the knots. Assume that the initial
framing of $K$ was $m$. Observe that the framing of the
transverse push-off $K^+$ with respect to $T_1$ is given by $m-1$. 
The isotopy of knots $K^+_s$ does not change the framing at all. Hence
at this point we constructed a transverse isotopy connecting $(K,m)$
and $(K',m-1)$.
Since the slope of $K'$ equals $\infty$ we can apply the twists
$\tau_\gamma^{m-1}$ around $K'$ such that we end up with a zero-framed knot $K'$ and we are done.
\end{proof}
\begin{ex*}
Consider two copies of the open book $(D^2, \id)$ supporting $S^3$ with
the standard contact structure $\xi_{st}$. It is easy to see that the
result of the fibre connected sum along the only binding components yields
$S^1 \times S^2$ with its standard contact structure. Using the
description of a compatible open book for the binding sum in Theorem~\ref{prop:open book supporting the binding sum} we obtain the standard open
book description of $S^1 \times S^2$ given by an annulus with trivial monodromy.
\end{ex*}
%
%
\subsection{Applications}
%
%
We finish the first part of the present paper with a few applications of the open book description given in Theorem~\ref{prop:open book supporting the binding sum}.
%
%
\subsubsection{Tight contact structures on $T^3$}
\label{sec:open books on torus}
%
%
Let $(\theta_1, \theta_2, \varphi)$ denote coordinates on the $3$-dimensional
torus $T^3$ and consider the tight contact structure $\xi_n$ given by the
kernel of the contact form $\cos(n \varphi)\thinspace
d{\theta_1} + \sin(n \varphi) \thinspace d{\theta_2}$. The contact structures $\xi_n$ provide a
complete list of tight contact structures on $T^3$ (cf.\ \cite{MR1487723})and can also be
described in the following way. Take $2n$ copies of the open book $(S^1 \times
[0,1],\id)$, which is an open book compatible with the standard 
contact structure on $S^1 \times S^2$, and then perform the $2n$-fold binding
sum in the obvious way. Now using Theorem~\ref{prop:open book supporting
the binding sum} we are able to translate the above construction of $(T^3,\xi_n)$ 
into compatible open books.
These open books for $(T^3,\xi_n)$ were first computed by van Horn-Morris~\cite{vanHornMorris-Thesis} using different methods than the ones presented
here.
%
%
\subsubsection{Full Lutz twist along binding component}
\label{sec:open book for lutz}
%
%
Consecutively performing the binding sum with
two copies of the open book $(S^1 \times [0,1],\id)$ has the effect of a full Lutz twist 
along the binding component. Again using Theorem~\ref{prop:open book supporting the binding sum}
we are able to compute a compatible open book. 
One can show that this open book is stably
equivalent to the compatible open book constructed in \cite{2009arXiv0905.0986O}.
Obviously we can compute the effect of a regular Lutz twist in the same fashion.
%
\subsubsection{Recovering Giroux torsion}
%
Let $(M,\xi)$ be a contact-$3$-manifold with non-zero Giroux torsion, i.e.~if
we choose $(\theta_1,\theta_2)$ to be coordinates on $T^2$ there exists an
embedding of the contact manifold
\[
			 \Big(T^2 \times [0,2\pi],\xi_{2\pi} = \ker\big(\cos( t)\thinspace
			d{\theta_1} + \sin( t) \thinspace d{\theta_2}\big)\Big)
\]
into $(M,\xi)$.
So far, there was no way to recover Giroux-torsion in
the language of open books. We approach this question by computing a certain 
compatible open book for $(M,\xi)$.
\parskip 0pt

Consider the complement $(M,\xi) \setminus \big(T^2 \times [0,2\pi],\xi_1
\big)$ of the Giroux-domain in $(M,\xi)$.
The boundary of $(M,\xi) \setminus \big(T^2 \times [0,2\pi],\xi_1 \big)$
consists of two pre-Lagrangian tori which are foliated by an $S^1$-family of
closed curves. Collapsing these tori, in the sense of \S\ref{sec:def fibre sum}, gives rise to a new closed contact manifold $(M',\xi')$
with two distinguished transverse knots $K_0$ and $K_1$. In this particular
case we decorate these knots with the framing corresponding to the $\theta_1$
coordinate. Let $(\Sigma',\phi')$ denote a compatible open book decomposition
of $(M',\xi')$ such that $K_0$ and $K_1$ are part of the binding
$\partial\Sigma'$. Let $m,n \in \Z$ be the above framings (induced by
$\theta_1$) expressed with respect to the page-framing induced by the open book
$(\Sigma',\phi')$.
\parskip 0pt

Observe that if we take two copies of the open book $(S^1\times[0,1],\id)$, perform the binding sum along $S^1 \times \{0\}$ in each copy of $(S^1\times[0,1],\id)$ and blow up the two remaining components corresponding to $S^1 \times \{1\}$ in each copy we end up with $\big(T^2 \times [0,2\pi],\xi_{2\pi} \big)$. Hence performing the $2$-fold binding sum of $(\Sigma',\phi')$ with $(S^1\times[0,1],\id) \boxplus (S^1\times[0,1],\id)$ along $K_0$ and the first copy of $S^1 \times \{1\}$ and along $K_1$ and the second copy of $S^1 \times \{1\}$ actually gives us a description of $(M,\xi)$ which, using Theorem~\ref{prop:open book supporting the binding sum}, may be translated into an open book.
%
%
\subsubsection{Surface bundles with invariant dividing set}
%
Let $\Sigma$ be a closed oriented surface containing a collection $\Gamma \subset \Sigma$ of oriented, mutually disjoint, embedded circles. Suppose there is a choice of orientations on the regions $\Sigma \setminus \Gamma$ which is coherent with the orientation of $\Gamma$. Let $\Sigma^+$ and $\Sigma^-$ denote the collections of positive and negative oriented regions of $\Sigma\setminus\Gamma$. We will refer to $(\Sigma,\Gamma)$ as an \textbf{abstract convex surface}. Let $\phi\co \Sigma \to \Sigma$ denote a diffeomorphism that restricts to the identity in a neighbourhood $N(\Gamma)$ of the abstract dividing set $\Gamma$. Write $\pi$ for the projection from the mapping torus $\Sigma(\phi)$ to the circle. Note that there is a natural contact structure $\xi_\Gamma$, such that for each $\theta \in S^1$ the fibre $\pi^{-1}(\theta) \cong \Sigma$ is a convex surface with dividing set $\Gamma$. 
\parskip 0pt

Write $\phi^\pm$ for the restrictions of $\phi$ to $\Sigma^\pm$. Then we may identify the contact manifold $(\Sigma(\phi),\xi_\Gamma)$ with the fibre sum of $(\Sigma^+,\phi^+)$ and $(\Sigma^-,\phi^-)$ (where the latter objects are understood as open books), i.e.\ we have
\[
	(\Sigma(\phi),\xi_\Gamma) = (\Sigma^+,\phi^+) \boxplus (\Sigma^-,\phi^-).
\]
With the help of Theorem~\ref{prop:open book supporting the binding sum} this identification may be translated into an open book.
%
%
\section{Fibre connected sum along multi-sections}
\label{sec:multisections}
%
%
In this section we try to approach the following question. Assume we are given
two knots $K_0$ and $K_1$ in the $3$-dimensional sphere $S^3$ which are
braided over the unkot $U \subset S^3$. Furthermore we assume the knots to have the same
braid index, $n \in \mathbb{N}$ say. Now recall the standard open book
description of $S^3$ with binding the unknot and pages diffeomorphic to the
$2$-disc and note that each of the knots $K_0$ and $K_1$ provide an
\textbf{$n$-fold section of of the open book}, i.e. each of the knots
intersects every page transversely and exactly $n$ times. A representation of the knots
$K_0$ and $K_1$ as braids endows them with a natural framing given 
by the the blackboard framing. Taking two copies of $(D^2,\id)$ we
can perform the fibre connected sum along $K_0$ and $K_1$ and ask for a
description of the resulting open book
\begin{equation}
\label{eqn:multisum}
			(\Sigma,\phi) := (D^2,\id) \#_{K_0, K_1} (D^2,\id).
\end{equation}
Obviously the page $\Sigma$ will be the $n$-fold connected sum of the two
original pages. However it is not clear what the monodromy $\phi$ looks like.
This question will be settled in the following two subsections.
\parskip 0pt

We assume that the reader is familiar with the basic notions of braid theory. For 
a brief introduction we point the reader to \cite{MR1414898}.
%
\subsection{Monodromy corresponding to a pair of crossings}
%
Before we dive into the description of $(\Sigma,\phi)$ we first
set up some notation and define a {relative version of the
fibre connected sum}. For a,
not necessarily connected, manifold $M$ with non-empty boundary $\partial M$
and two collections of properly embedded, oriented, framed arcs
$\boldsymbol a = \{ a_1,\ldots, a_k \}$ and $\boldsymbol a' = \{ a'_1,\ldots,a'_k \}$ with neighbourhoods
$N_{\boldsymbol a}$ and $N_{\boldsymbol a'}$ we denote by $\#_{\boldsymbol a}M$ the manifold
\[
	\#_{\boldsymbol a}M := \big( M \setminus (N_{\boldsymbol a} \cup N_{\boldsymbol a'}) \big)
	/_{ _{\partial N_{\boldsymbol a}} \sim_{ \partial N_{\boldsymbol a'}} },
\]
where we identify as follows: for $i=1,\ldots,k$ the framing together with the orientation induce identifications of both components $N_{a_i} \subset N_{\boldsymbol a}$ and $N_{a_i'} \subset N_{\boldsymbol a'}$ with $[0,1] \times D^2$. Now we identify $(t,\theta) \in [0,1]\times \partial D^2 \subset \partial N_{a_i}$ with $(t,-\theta) \in [0,1]\times \partial D^2 \subset \partial N_{a_i'}$.
\begin{figure}[h] 
\center
\includegraphics{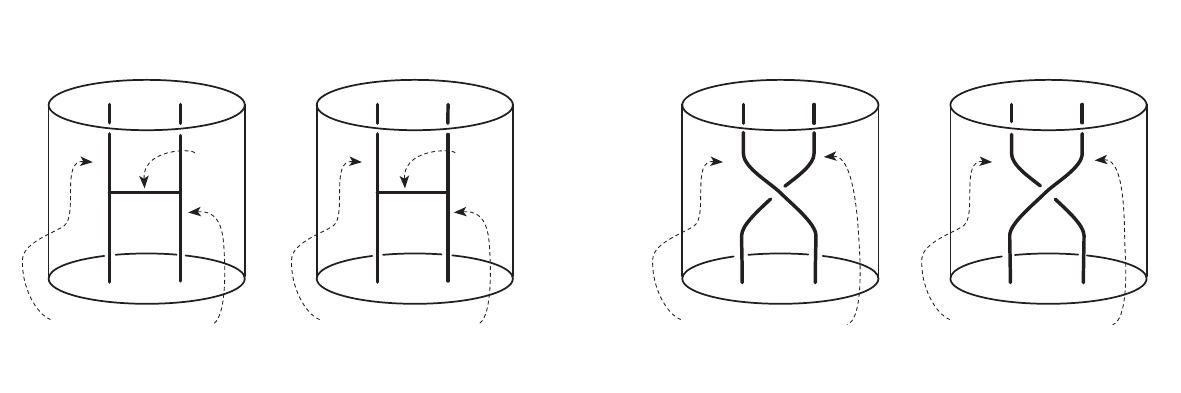}
\put (-264,57) {$\sqcup$}
\put (-82,57) {$\sqcup$}
\put (-323,15){$a_1$} \put (-292,15){$a_2$}
\put (-281,66){$c$}
\put (-243,15){$a^\prime_1$}
\put (-214,15){$a^\prime_2$}
\put (-205,66){$c^\prime$}
\put (-137,15){$b_1$}
\put (-108,15){$b_2$}
\put (-60,15){$b^\prime_1$}
\put (-30,15){$b^\prime_2$}
\caption{The arcs $\boldsymbol a,\boldsymbol a', \boldsymbol b,\boldsymbol b', c,c' \subset \big( D
\times [0,1] \big) \sqcup \big( D' \times [0,1] \big)$.}
\label{fig:arcs a b}
\end{figure}
\parskip 0pt

From now on let $M$ be the disjoint union of $D \times [0,1]$ and $D' \times
[0,1]$, in symbols
\[
		M = \big( D \times [0,1] \big) \sqcup \big( D' \times [0,1] \big),
\]
where $D$ and $D'$ respectively denote a copy of the $2$-disc
$D^2$.  
Consider the two sets of properly embedded, framed (by the blackboard-framing)
arcs $\boldsymbol a,\boldsymbol a', \boldsymbol b, \boldsymbol b'\subset M$ indicated in Figure~\ref{fig:arcs a
b}. Understand these arcs as oriented upwards and consider the corresponding manifolds $\#_{\boldsymbol a}M$ and $\#_{\boldsymbol b}M$. Furthermore let $K,K' \subset \#_{\boldsymbol a}M$ denote the framed knots indicated in Figure~\ref{fig:K and K'}. The
framings are measured with respect to $\Sigma' \subset \#_{\boldsymbol a}M$,
the genus-$1$ surface with two boundary components obtained by the $2$-fold
connected sum of $D$ and $D'$. Observe that $\#_{\boldsymbol a}M$ is naturally diffeomorphic to $\Sigma' \times [0,1]$.
\begin{figure}[h] 
\center
\includegraphics{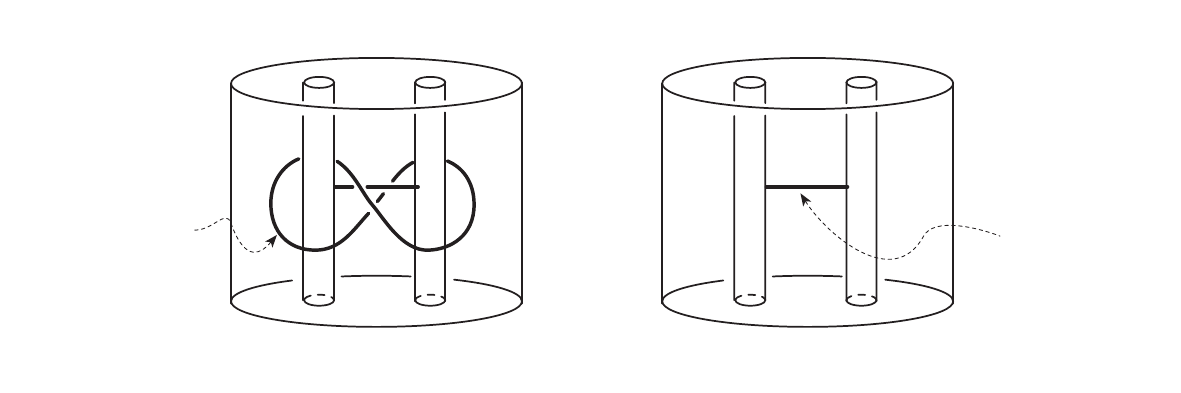}
\put (-53,37){{$(K,0)$}}
\put (-323,37){{$(K',+1)$}}
\put (-180,57){$\sim$}
\caption{The knots $K$ and $K'$ sitting in $\#_{\boldsymbol a}M$. Framings are
measured with respect to $\Sigma' \subset \#_{\boldsymbol a}M$.}
\label{fig:K and K'}
\end{figure}
\begin{lem}
\label{lem:aM to bM}
Denote by $(\#_{\boldsymbol a}M)(K,K')$ the result of surgery along $K_0,K_1$
with respect to their framings (cf.~Figure~\ref{fig:K and K'}). Then we have
\[
			 \#_{\boldsymbol b}M  \cong 	( \#_{\boldsymbol a}M ) (K,K').
\]\end{lem}
\begin{proof}
Let us give an explicit description of $M = \big( D \times [0,1] \big)
\sqcup \big( D' \times [0,1] \big)$ embedded in $\R^3$ with coordinates
$(x,y,z)$: understand $D$ and $D'$ as unit $2$-discs in the $xy$-plane centred at
the points $(0,2)$ and $(0,-2)$. Identify $\boldsymbol a, \boldsymbol a'$ with
$\{ (\pm\frac{1}{2} , 2)\} \times [0,1]$ and $\{ (\pm\frac{1}{2} , -2)\} \times
[0,1]$. Denote by $c,c'$ the arcs given as $[-\frac{1}{2},\frac{1}{2}] \times
\{ \pm 2 \} \times \{\frac{1}{2}\}$. Then
$\#_{\boldsymbol a}M$ is given as the quotient 
\[  
 	 \Big( \big( D \times [0,1] \big) \setminus
	 N_{\boldsymbol a} \Big) \cup \Big( \big( D' \times [0,1] \big) \setminus
	 N_{\boldsymbol a'} \Big) / \sim_{\partial N_{\boldsymbol a}} ,
\] 
where we identify points
$(x,y,z) \in \partial N_{\boldsymbol a}$ with their mirror image $(x,y,-z) \in
\partial N_{\boldsymbol a'}$. Note that the pair of arcs $c,c'$ descends to the
closed curve $K \subset \#_{\boldsymbol a}M$.
\parskip 0pt

Denote by $H \subset M$ a neighbourhood of the graph $(a_1 \cup a_2) \cup c$. 
Choose $H'$ to be the reflection of $H$ with
respect to the $xz$-plane and note that $H'$ provides a neighbourhood of the graph $(a'_1 \cup a'_2 ) \cup c'$. Observe that the result $( \#_{\boldsymbol a}M )
(K)$ of zero-surgery along $K$ is given by
\[
	\Big( \big( D \times [0,1] \big) \setminus H
	\Big) \cup \Big( \big( D' \times [0,1] \big) \setminus H' \Big) /
	\sim_{\partial H} ,
\] 
where we identify a point $(x,y,z) \in \partial H$ with its mirror $(x,-y,z)
\in \partial H'$.
\parskip 0pt

Let us now perform the surgery along $K'$. Isotope $K'$ such that it lies on
$\partial H$ sitting inside of $\#_{\boldsymbol a}M$. Note that
the framing of $K'$ and the framing induced by $\partial H$ agree.
Let $\nu K' = (-\varepsilon,\varepsilon) \times S^1$ denote a small open
neighbourhood of $K'$ in $\partial H$. Remove a
neighbourhood $N_{K'} \subset ( \#_{\boldsymbol a}M ) (K) $ of $K'$ and
observe that topologically the complement of $N_{K'}$ is given by
\begin{equation}
\label{eqn:complement of K_1}
	\Big( \big( D \times [0,1]
	\big) \setminus H \Big) \cup \Big( \big( D' \times [0,1] \big) \setminus H' \Big) / \sim_{\partial H \setminus \nu K'} ,
\end{equation}
where we just identify points $(x,y,z) \in \partial H \setminus \nu K'$ with
their mirror $(x,-y,z) \in \partial H'$. We will glue back the surgery torus
$S^1 \times D^2$ in two steps. Take $[0,\pi] \times
D^2 \subset S^1 \times D^2$ (where we identify $S^1 \equiv \R/2\pi\Z$) and
attach it along $[0,\pi] \times \partial D^2$ to the closure of the neighbourhood
$\nu K' \subset \partial H$, which we identify with $[-\varepsilon,\varepsilon]
\times S^1$.
Simultaneously attach $[\pi,2\pi] \times D^2$
along  $[\pi,2\pi] \times \partial D^2$ to the mirror image of $\nu K'$ on
$\partial H'$. We can actually picture this to be done
inside of $H$ and $H'$ respectively.  
Observe that the two pieces $\big([0,\pi] \times D^2\big), \big([\pi,2\pi] \times D^2\big)$, attached to the complement described in description (\ref{eqn:complement of K_1}) above, really descend to a solid torus. Moreover observe that we can understand the boundary of $N_{\boldsymbol b}$ as decomposes as $(\partial H \setminus \nu K')\cup (\{0,\pi\} \times D^2)$. Hence gluing back the surgery torus to the space given in (\ref{eqn:complement of K_1}) gives
\[  
 	 \Big( \big( D \times [0,1] \big) \setminus
	 N_{\boldsymbol b} \Big) \cup \Big( \big( D' \times [0,1] \big) \setminus
	 N_{\boldsymbol b'} \Big) / \sim_{\partial N_{\boldsymbol b}},
\]
which describes $\#_{\boldsymbol b}M$. This completes the proof.
\end{proof}
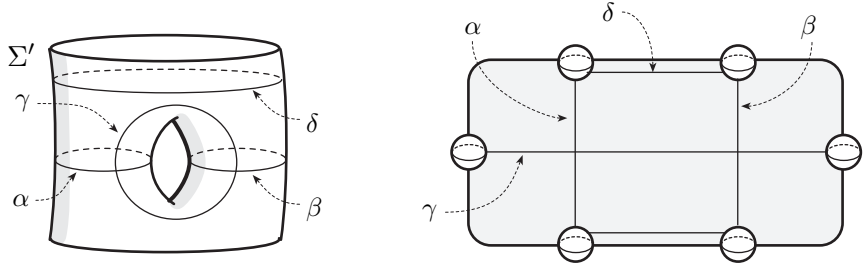
\begin{figure}[h] 
\center
\includegraphics{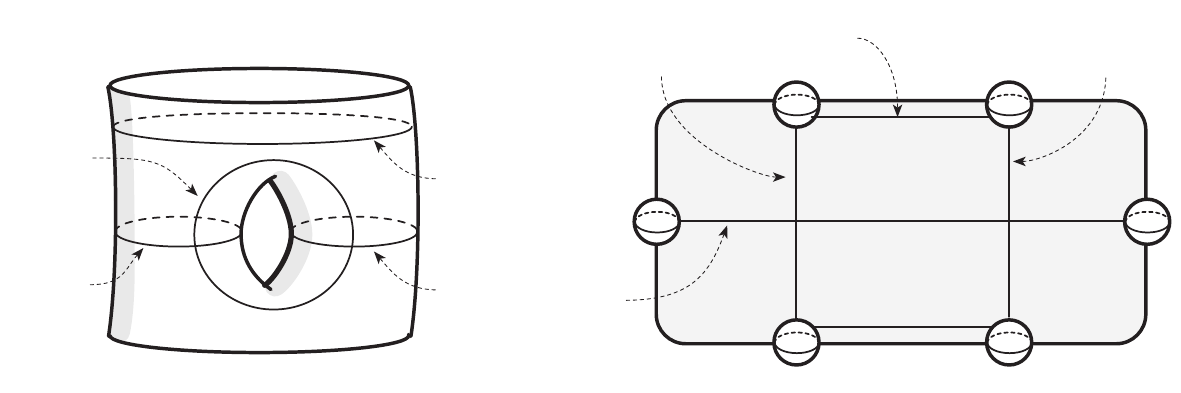}
\put(-326,82){\large{$\Sigma'$}}
\put(-213,57){$\delta$}
\put(-213,25){$\beta$}
\put(-324,29){$\alpha$}
\put(-323,67){$\gamma$}
\put(-101,99){$\delta$}
\put(-153,94){$\alpha$}
\put(-26,94){$\beta$}
\put(-169,25){$\gamma$}
\caption{Two identifications of $\Sigma'$ with the curves
$\alpha,\beta,\gamma,\delta$ used in Lemma
\ref{lem:local mon}.}
\label{fig:sigma'}
\end{figure}
Recall that $\Sigma'$ denotes the genus-$1$ surface with two boundary components
understood as obtained by the $2$-fold connected sum of $D$ and $D'$.
Recall further $\#_{\boldsymbol a}M$ is naturally isomorphic to $\Sigma' \times
[0,1]$. We would now like to express the surgery along $K,K' \subset
\#_{\boldsymbol a}M$ in the above description of $\#_{\boldsymbol b}M$ as a
sequence of $\pm 1$-surgeries along certain curves on $\Sigma'$, where the
framing is measured with respect to $\Sigma'$.
\begin{lem}
\label{lem:local mon}
Let $\alpha,\beta,\gamma,\delta \subset \Sigma'$ denote the curves described in Figure~\ref{fig:sigma'}. Then 
setting $\psi = (\tau_\alpha \tau_\beta \tau_\gamma)^2 (\tau_{\delta})^{-1}$ we have
\[
	(\#_{\boldsymbol a}M) (K,K') =	\Big( \big(  \Sigma' \times [0,1] \big) \cup \big(  \Sigma' \times [2,3] \big) \Big) / \sim_\psi,
\]
where we identify $(p,1)$ with $(\psi(p),2)$ for each $p \in \Sigma'$.
\end{lem}
\begin{proof}
Recall that in Lemma \ref{lem:aM to bM} we identified $(\#_{\boldsymbol a}M)
(K,K')$ with $\#_{\boldsymbol b}M$. Observe that the latter space admits the
structure of a $\Sigma'$-fibration which is induced by the projection on the
unit interval $[0,1]$. The map $\psi$ is actually a factorisation of the
monodromy of this fibration into Dehn twists. A little caution is needed:
unfortunately we are actually computing the inverse of $\psi$,
since in our computations we push arcs from the top to the bottom, not the other
way round.
\begin{figure}[htp] 
\center
\includegraphics{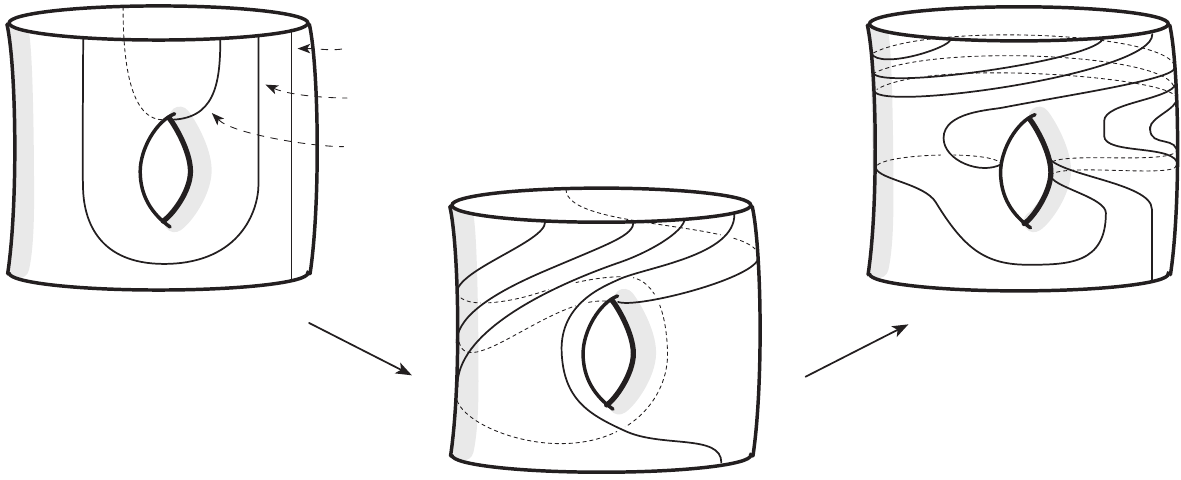}
\put(-238,121){$A$}
\put(-238,108){$B$}
\put(-238,93){$C$}
\put(-255,28){$\psi^{-1}$}
\put(-90,28){$P$}
\caption{A cut system for $\Sigma'$ and its image under $P \circ \psi^{-1}$.}
\label{fig:cut system}
\end{figure}
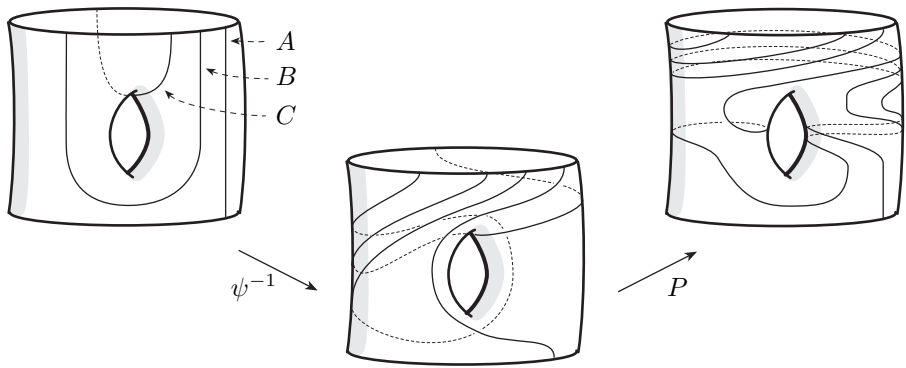
\parskip 0pt

Let $A,B,C \subset \Sigma'$ denote the cut system given in the left part of
Figure~\ref{fig:cut system}. The images of this cut system under $\psi^{-1}$ are given in the
middle part of Figure~\ref{fig:cut system}. The images were computed as follows: recall that $\#_{\boldsymbol a}M$ did correspond to $\Sigma' \times [0,1]$ which we
understand as obtained by thicken up the shaded area in Figure~\ref{fig:sigma'} (or Figure~\ref{fig:Kirby K,K'} respectively). A description of
the knots $K,K'$ with respect to this perspective is given in Figure~\ref{fig:Kirby K,K'}. Understand the result of surgery $(\#_{\boldsymbol a}M)(K,K')$ on
$K,K'$ as embedded in the Kirby diagram given in Figure~\ref{fig:Kirby K,K'}.
We can now recover the cut system, chosen above, in the Kirby diagram and
manipulate it using Kirby calculus. The actual computations are given in
Figures~\ref{fig:image A}, \ref{fig:image B} and \ref{fig:image C} on
p.~\pageref{fig:image A}ff at the end of the paper.
\parskip 0pt

We could now just compare these images with the ones under the inverse of $(\tau_\alpha
\tau_\beta \tau_\gamma)^2 (\tau_{\delta})^{-1}$ and conclude that
both agree up to isotopy, showing that $\psi = (\tau_\alpha \tau_\beta
\tau_\gamma)^2 (\tau_{\delta})^{-1}$.
\begin{figure}[h] 
\center
\includegraphics{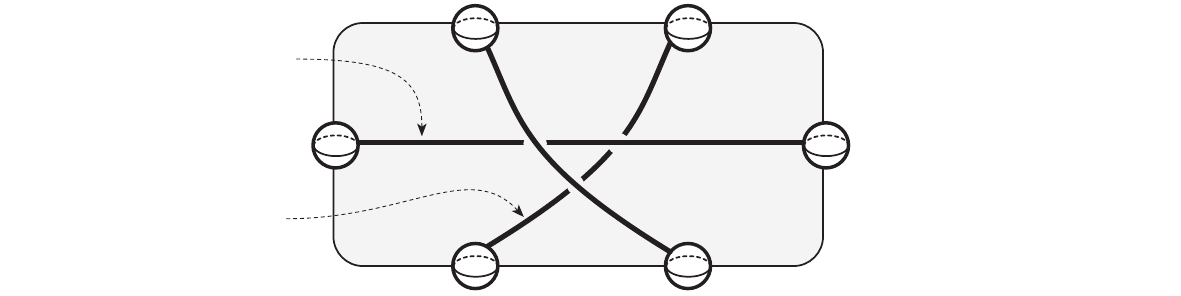}
\put (-282,70){$(K,0)$}
\put (-297,24){$(K',+1)$}
\caption{A Kirby diagram showing $K,K'$ and $\Sigma'$.}
\label{fig:Kirby K,K'}
\end{figure}
However the usual way to compute the factorisation of a monodromy map of a
compact surface is by by reducing it to the case of a self-diffeomorphism of a
disc with punctures, cf.~\cite{MR1414898}.
Referring to Figure~\ref{fig:image of a} (on
p.~\pageref{fig:image of a}) the image of $\alpha$ under $\psi^{-1}$ is given
as $\beta $. Set
\[
	P = \tau_\gamma \thinspace \tau_\beta \thinspace
\tau_\alpha \thinspace \tau_\gamma
\]
and note that $P$ maps $\beta$ to $\alpha$. Therefore $P \circ \psi^{-1}$ fixes
the curve $\alpha$ and hence can now be interpreted as a self-diffeomorphism
of the $3$-fold punctured disc $D_3$ obtained by cutting $\Sigma^\prime$ along
$\alpha$. Note that $A,B \subset \Sigma'$ descends to a cut system of $D_3 =
\Sigma^\prime\setminus\alpha$. Therefore all the data of  $P \circ \psi^{-1}$
is encoded in the images of $A,B \subset \Sigma'$.
The images of $A,B \subset \Sigma'$ under $P \circ \psi^{-1}$ are given in the
right part of Figure~\ref{fig:cut system} (cf.\ Figure~\ref{fig:A under P} and
Figure~\ref{fig:B under P} on p.~\pageref{fig:A under P} and \pageref{fig:B under P} for the
actual computations). We conclude that we have
\[
			P \circ \psi^{-1} = \tau_\alpha^{-1} \thinspace \tau_\beta^{-1} \thinspace
			\tau_\delta.
\]
Therefore $\psi^{-1}$ is given by $(\tau_\alpha^{-1} \thinspace \tau_\beta^{-1}
\thinspace \tau_\delta) \circ P^{-1}$, which, computing the inverse, is exactly
what we intended to show.
\end{proof}
\begin{figure}[h] 
\center
\includegraphics{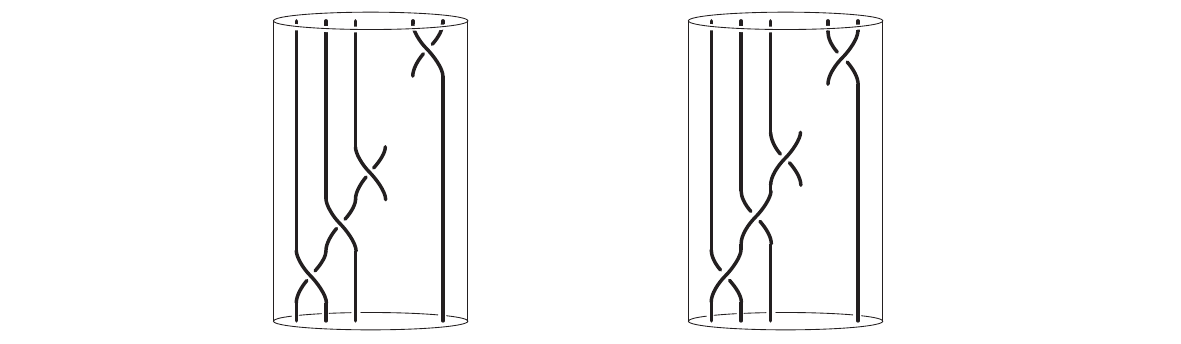}
\put (-112,67){$\dots$}
\put (-280,80){$S^-$}
\put (-160,80){$S^+$}
\put (-232,67){$\dots$}
\caption{Description of the standard braids.}
\label{fig:stdandard braids}
\end{figure}
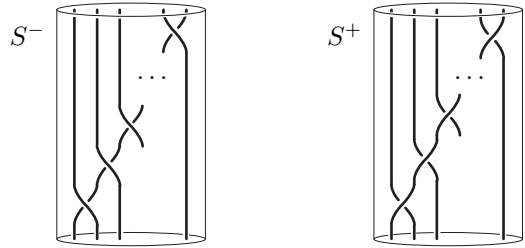
With this in hand we are able to compute the monodromy $\phi$ for
the case that $K_0$ and $K_1$ are chosen among the standard braids $S^+, S^-$
given in Figure~\ref{fig:stdandard braids}.
\begin{prop}
\label{cor:monodromy of standard braids}
Let $\alpha_1,\ldots,\alpha_n,\gamma_1,\ldots,\gamma_{n-1},\delta^{(\prime)}_{1,2},\ldots,\delta^{(\prime)}_{n-1,n} \subset \Sigma$ denote the curves indicated in Figure~\ref{fig:monodromy curves}. Then we have
\begin{list}{}{}
	\item[(i)] $\phi = \prod_{i=1}^{n-1} \tau_{\gamma_{i}}\thinspace \tau_{\alpha_i} \thinspace \tau_{\alpha_{i+1}}\thinspace \tau_{\gamma_{i}} $, for the pair of knots $(K_0,K_1)=(S^+,S^+)$,
	\item[(ii)] $\phi = \prod_{i=1}^{n-1} ( \tau_{\gamma_i} \thinspace \tau_{\alpha_i} \thinspace\tau_{\alpha_{i+1}})^2 \thinspace(\tau_{\delta_{i,i+1}})^{-1}  $, for the pair of knots $(K_0,K_1)=(S^-,S^+)$ and 
		\item[(iii)] $\phi = \prod_{i=1}^{n-1} \tau_{\alpha_i} \thinspace\tau_{\alpha_{i+1}} \thinspace( \tau_{\gamma_i} \thinspace \tau_{\alpha_i} \thinspace\tau_{\alpha_{i+1}})^2 \thinspace(\tau_{\delta_{i,i+1}})^{-1} \thinspace(\tau_{\delta'_{i,i+1}})^{-1}  $, for the pair of knots $(K_0,K_1)=(S^-,S^-)$.
\end{list} 
\end{prop} 
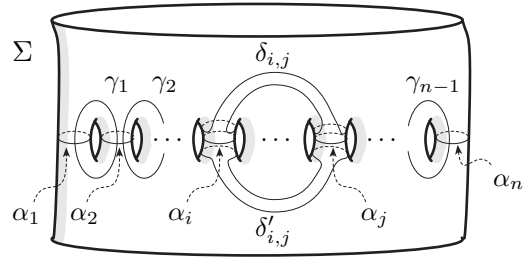
\begin{figure}[h] 
\center
\includegraphics{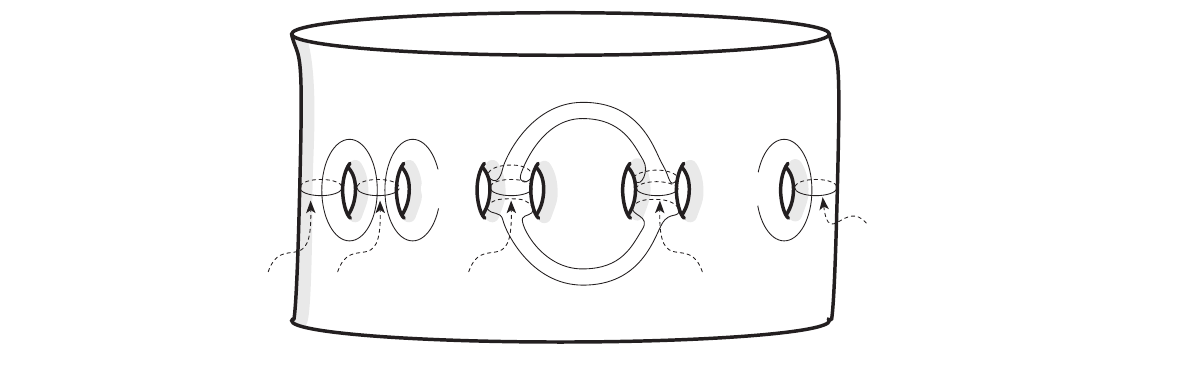}
\put (-272,82){\large{$\Sigma$}}
\put (-180,82){$\delta_{i,j}$}
\put (-180,16){$\delta^\prime_{i,j}$}
\put (-237,72){$\gamma_1$}
\put (-219,72){$\gamma_2$}
\put (-123,72){$\gamma_{n-1}$}
\put (-272,22){$\alpha_1$}
\put (-250,22){$\alpha_2$}
\put (-213,22){$\alpha_i$}
\put (-140,22){$\alpha_j$}
\put (-90,35){$\alpha_{n}$}
\put (-219,52){$\ldots$}
\put (-178,52){$\ldots$}
\put (-138,52){$\ldots$}
\caption{The curves used in
Corollary
\ref{cor:monodromy of standard braids} and Proposition \ref{prop:multisections}.}
\label{fig:monodromy curves}
\end{figure}
\begin{proof}
We start by proving the second part of the statement. Let $n \in \N$ be the
braid index of $K_0$ and $K_1$ respectively. Consider two copies $L_0,L_1$ of
the trivial braid of $n$-strands describing an $n$-component unlink  and
perform a fibre connected sum for each pair of unknots. By applying Lemma
\ref{lem:aM to bM} exactly $n-1$ times we can turn this $n$-fold fibre
connected sum into the fibre connected sum along $S^+$ and $S^-$. Keeping track
of the change of monodromy completes the proof of the second part.
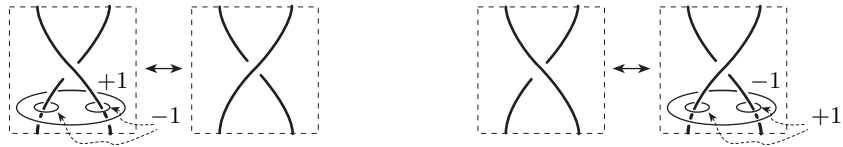
\begin{figure}[h] 
\center
\includegraphics{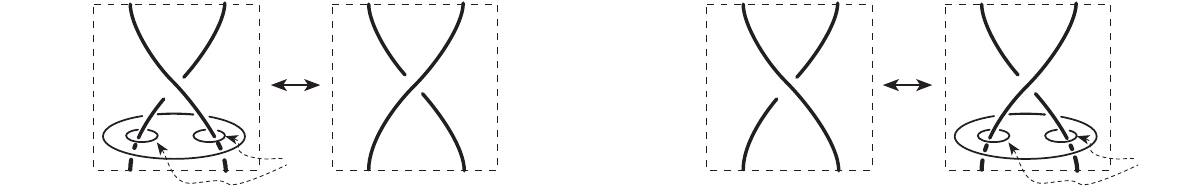}
\put(-11,9) {\small{$+1$}}
\put(-34,22) {\small{$-1$}}
\put(-261,9) {\small{$-1$}}
\put(-281,22) {\small{$+1$}}
\caption{Switching between positive and negative
crossings.}
\label{fig:crossing sign}
\end{figure}
\parskip 0pt

In Lemma \ref{lem:aM to bM} we are considering the result of fibre summing a
negative crossing (the arcs $a_1,a_2$) with a positive one (the arcs
$b_1,b_2$). Perform a surgery as indicated in the left part of Figure~\ref{fig:crossing sign} (or right part of Figure~\ref{fig:crossing sign} respectively) and observe that we turned the negative (or positive respectively)
crossing into a positive (or negative respectively) crossing. This actually is
nothing but a certain Rolfsen twist. However by performing one of these
surgeries we can always set up the situation for which Lemma \ref{lem:aM to bM}
applies. Translating the surgery into the language of Dehn twists sets the way
to compute the monodromy for the remaining cases and we are done.
\end{proof}
\begin{figure}[h] 
\center
\includegraphics[scale=0.9]{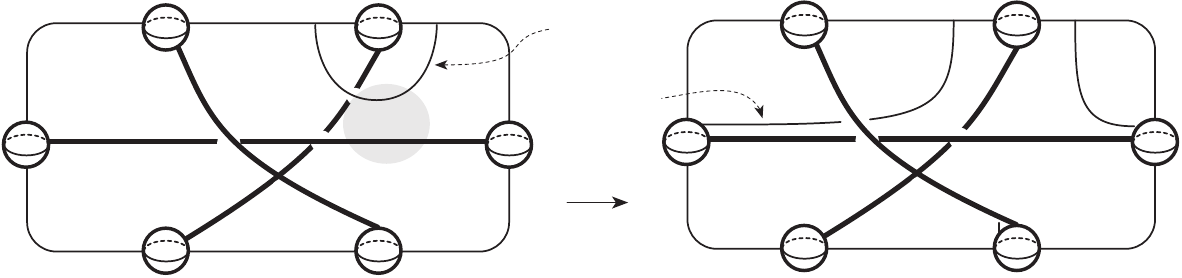}  
\put (-163,62){\small$A$}
\put (-168,45){\small$\psi^{-1}(A)$}
\caption{The shaded area indicates that a
handle slide is performed.}
\label{fig:image A}
\end{figure}
\subsection{Final computation of the monodromy}
We almost have everything in place to compute the monodromy map $\phi$ of the
fibre sum along multi-sections (see description (\ref{eqn:multisum}) at the
beginning of the section). The last ingredient is the following normal
form for a braided knot $K$. Let $B$ be a braid representation of $K \subset
S^3$ with braid index $n\in \N$ and let $S=S^+$ be the positive standard braid
indicated in the right part of Figure~\ref{fig:stdandard braids}.
By an isotopy of $K$ we may assume that the permutation induced by $B$ is given
by $(n \ 1 \ \ldots \ n-1)$. Therefore $P = B \ast S^{-1}$ describes a pure
braid for which we obviously have $B = P \ast S$. Here ``$\ast$'' denotes the
composition of braids in the braid group.
\parskip 0pt

According to \cite{MR1414898} one can assign a pure braid to a diffeomorphism
of the $n$-fold punctured disc, equal to the identity near the boundary, and
vice versa. Let $\phi_K$ denote the map corresponding to the pure braid $P$
(which itself, by the consideration above, is induced by $K$ ). Note that the map
$\phi_K$ encodes all information about $K$.
\parskip 0pt

Let us return to the open book description $(\Sigma,\phi)$ of the fibre sum
along $K_0$ and $K_1$ (cf.\ (\ref{eqn:multisum}) above). Denote by $\phi_{K_0}$
and $\phi_{K_1}$ the maps associated to the knots $K_0$ and $K_1$ as described
above. These maps trivially extend to $\Sigma$. Together with the first part of
Proposition~\ref{cor:monodromy of standard braids} we finally obtain the
following description of $\phi$.
\begin{thm}
\label{prop:multisections}
The monodromy $\phi$ of the open book described in~(\ref{eqn:multisum}) is given
by
\[
 			\phi = \Big( \prod_{i=1}^{n-1}\tau_{\gamma_{i}} \thinspace \tau_{\alpha_i} \thinspace
			\tau_{\alpha_{i+1}}\thinspace \tau_{\gamma_{i}} \Big)  \circ \phi_{K_0}
			\circ \phi_{K_1},
\]
where $\phi_{K_0}$, $\phi_{K_1}$ are as described above and
$\alpha_1,\ldots,\alpha_n,\gamma_1,\ldots,\gamma_{n-1} \subset \Sigma$ denote
the curves indicated in Figure~\ref{fig:monodromy curves}. \qed
\end{thm}
\begin{figure}[htp] 
\center
\includegraphics[scale=0.9]{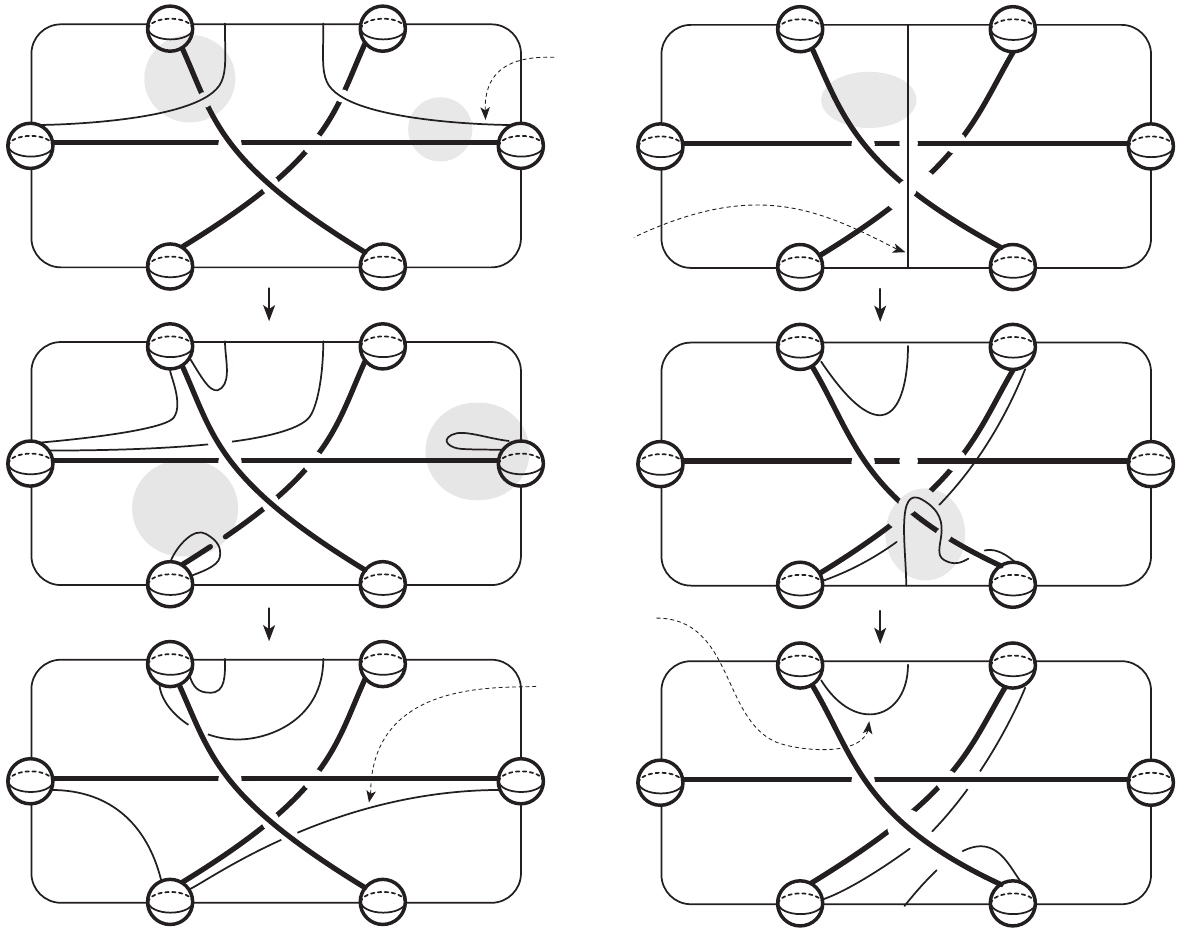}
\put (-163,225){\small$B$}
\put (-168,67){\small$\psi^{-1}(B)$}
\put (-153,178){\small$C$}
\put (-168,86){\small$\psi^{-1}(C)$}
\caption{The images of $B$ and $C$ under $\psi^{-1}$. The shaded areas indicate that an
isotopy or a handle slide is performed.}
\label{fig:image B}
\end{figure}
\begin{figure}[htp] 
\center
\includegraphics[scale=0.9]{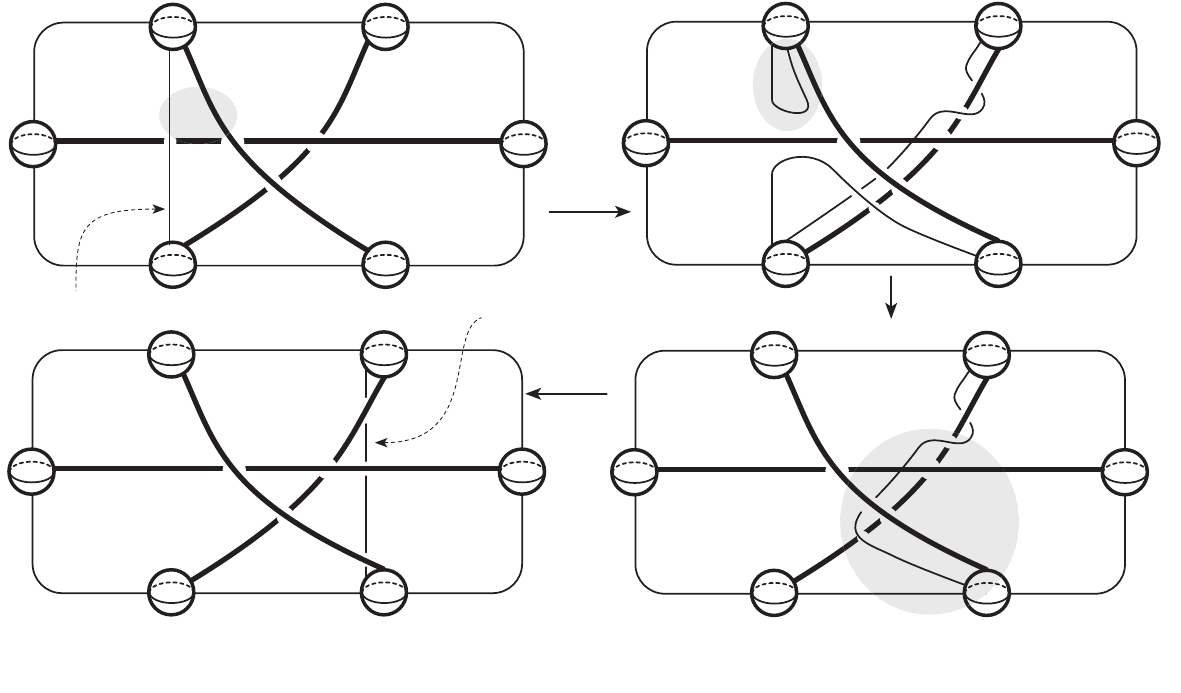} 
\put(-291,98){\small$\alpha$}
\put(-179,96){\small$\beta$}
\caption{Computation of the image of $\alpha$ under $\psi^{-1}$. The shaded areas indicate that an
isotopy or a handle slide is performed.}
\label{fig:image of a}
\end{figure}
\begin{figure}[htp] 
\center
\includegraphics[scale=1]{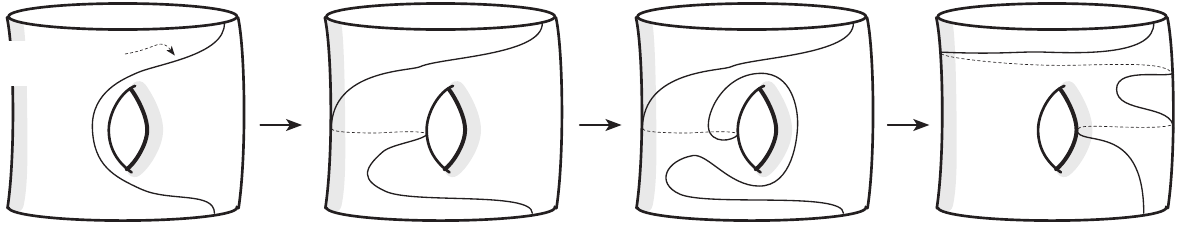}
\put(-340,45){\small$\psi^{-1}(A)$}
\put(-268,35){\small$\tau_\alpha \thinspace \tau_\gamma$}
\put(-177,35){\small$\tau_\gamma \thinspace \tau_\beta$}
\put(-82,35){\small${\simeq}$}
\caption{Computation of the image of $\psi^{-1}(A)$ under $P =\tau_\gamma
\thinspace \tau_\beta \thinspace\tau_\alpha \thinspace \tau_\gamma$.}
\label{fig:A under P}
\end{figure}
\begin{figure}[htp] 
\center
\includegraphics[scale=1]{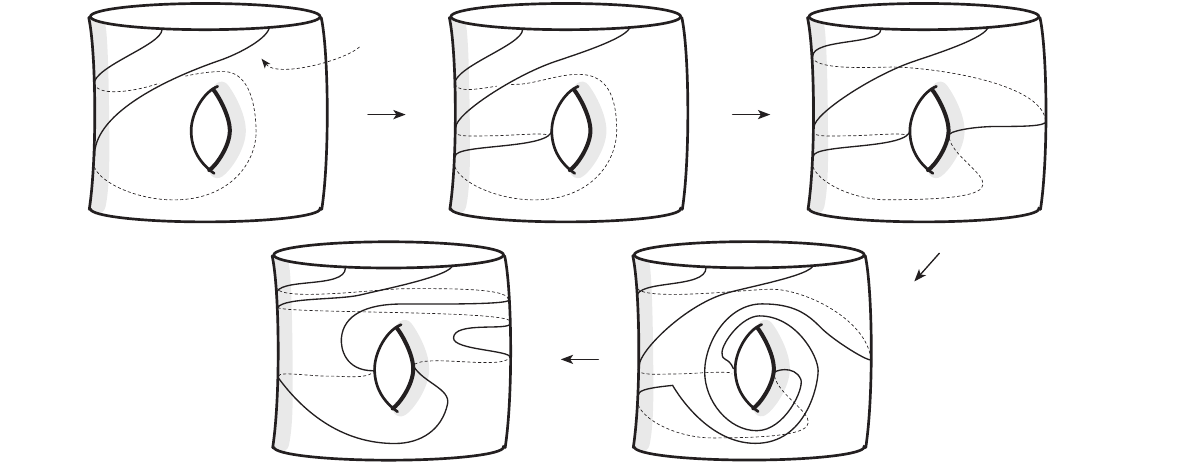}
\put(-245,128){\small$\psi^{-1}(B)$}
\put(-240,90){\small$\tau_\alpha \thinspace \tau_\gamma$}
\put(-130,90){\small$\tau_\beta$}
\put(-70,50){\small$\tau_\gamma$}
\put(-176,40){\small$\simeq$}
\caption{Computation of the image of $\psi^{-1}(B)$ under $P =\tau_\gamma
\thinspace \tau_\beta \thinspace\tau_\alpha \thinspace \tau_\gamma$.}
\label{fig:B under P}
\end{figure}
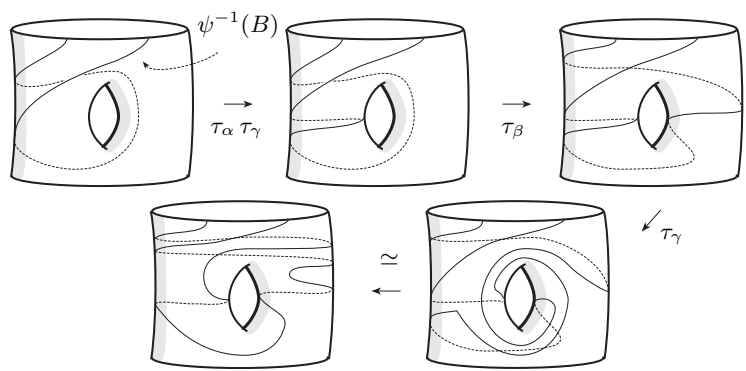
%
 
%
\end{document}